\documentclass{amsproc}%
\usepackage{amsfonts}
\usepackage{amsmath}
\usepackage{amssymb}
\usepackage{graphicx}%
\providecommand{\U}[1]{\protect\rule{.1in}{.1in}}
\theoremstyle{plain}

\newtheorem{corollary}{Corollary}

\newtheorem{example}{Example}

\newtheorem{lemma}{Lemma}

\newtheorem{theorem}{Theorem}
\numberwithin{equation}{section}
\begin{document}
\title[Path-connected Closure of Unitary Orbits]{Path-connected Closures of Unitary Orbits}
\author{Don Hadwin}
\address{University of New Hampshire}
\email{don@unh.edu}
\urladdr{}
\author{Wenjing Liu}
\curraddr{University of Massachusetts Boston}
\email{wenjingtwins87@gmail.com }
\dedicatory{Dedicated to Lyra}
\begin{abstract}
If $\mathcal{A}$ and $\mathcal{B}$ are unital C*-algebras and $\pi
:\mathcal{A}\rightarrow\mathcal{B}$ is a unital $\ast$-homomorphism, then
$\mathcal{U}_{\mathbb{B}}\left(  \pi\right)  ^{-}$ is the set of all $\ast
$-homomorphisms from $\mathcal{A}$ to $\mathcal{B}$ that are approximately
(unitarily) equivalent to $\pi.$ We address the question of when
$\mathcal{U}_{\mathbb{B}}\left(  \pi\right)  ^{-}$ is path-connected with
respect to the topology of pointewse norm convergence. When $\mathcal{A}$ is
singly generated and $\mathcal{B}=B\left(  \ell^{2}\right)  $, an affirmative
answer was given in \cite{H}; we extend this to the case when $\mathcal{A}$ is
separable. We also give an affirmative answer when $\mathcal{A}$ is AF and
$\mathcal{B}$ is a von Neumann algebra, $\mathcal{A}$ is ASH and $\mathcal{B}$
is a finite von Neumann algebra, or when $\mathcal{A}$ is homogeneous and
$\mathcal{B}$ is an arbitrary von Neumann algebra.

\end{abstract}
\maketitle

\section{Introduction}

In \cite{H} D. Hadwin proved that the norm closure of the unitary orbit of an
operator in $B\left(  \ell^{2}\right)  $ is path connected. In this paper we
address the problem of extending this result to representations of separable C*-algebras.

Suppose $\mathcal{A}$ and $\mathcal{B}$ are unital C*-algebras and
$\mathcal{A}$ is separable. We define \textrm{Rep}$\left(  \mathcal{A}%
,\mathcal{B}\right)  $ as the set of all unital $\ast$-homomorphisms from
$\mathcal{A}$ to $\mathcal{B}$ with the topology of pointwise norm
convergence. Suppose $\left\{  a_{1},a_{2},\ldots\right\}  $ is a norm dense
subset of the closed unit ball of $\mathcal{A}$. We define a metric
$d=d_{\mathcal{A},\mathcal{B}}$ by%
\[
d\left(  \pi,\rho\right)  =\sum_{m,n=1}^{\infty}\frac{1}{2^{m+n}}\left\Vert
\pi\left(  a_{n}\right)  -\rho\left(  a_{n}\right)  \right\Vert \text{ .}%
\]
Clearly, $d$ makes \textrm{Rep}$\left(  \mathcal{A},\mathcal{B}\right)  $ into
a complete metric space. When $\mathcal{B}$ is finite-dimensional,
\textrm{Rep}$\left(  \mathcal{A},\mathcal{B}\right)  $ is compact.

Let $\mathcal{U}_{\mathcal{B}}$ denote the group of unitary elements of
$\mathcal{B}$. If $\pi\in$ \textrm{Rep}$\left(  \mathcal{A},\mathcal{B}%
\right)  $, we define the \textbf{unitary orbit} $\mathcal{U}_{\mathcal{B}%
}\left(  \pi\right)  $ of $\pi$ by%
\[
\mathcal{U}_{\mathcal{B}}\left(  \pi\right)  =\left\{  U^{\ast}\pi\left(
\cdot\right)  U:U\in\mathcal{U}_{B}\right\}  \text{.}%
\]
If $T\in\mathcal{B}$ we define the unitary orbit $\mathcal{U}_{\mathcal{B}%
}\left(  T\right)  $ of $T$ by
\[
\mathcal{U}_{\mathcal{B}}\left(  T\right)  =\left\{  U^{\ast}TU:U\in
\mathcal{U}_{\mathcal{B}}\right\}  .
\]
It is clear that $\mathcal{U}_{\mathcal{B}}\left(  T\right)  $ corresponds to
$\mathcal{U}_{\mathcal{B}}\left(  \pi\right)  $ when $\pi$ is the identity
representation of the identity representation of $C^{\ast}\left(  T\right)  $.

In this paper we address the problem of when $\mathcal{U}_{\mathcal{B}}\left(
\pi\right)  ^{-}$ is path-connected in \textrm{Rep}$\left(  \mathcal{A}%
,\mathcal{B}\right)  $. In Section $2$ we discuss special paths in
$\mathcal{U}_{\mathcal{B}}\left(  \pi\right)  ^{-}$. In Section $3$ we provide
an affirmative answer (Theorem \ref{separable}) for the case when
$\mathcal{B}=B\left(  \ell^{2}\right)  $. We reduce the separable case to the
singly generated case by tensoring with the algebra $\mathcal{K}\left(
\ell^{2}\right)  $ of compact operators on $\ell^{2}$. In Section $4$ we give
an affirmative answer (Theorem \ref{AF}) when $\mathcal{A}$ is AF and
$\mathcal{B}$ has the property that $\mathcal{U}_{p\mathcal{B}p}$ is connected
for every projection $p\in\mathcal{B}$. We also give an affirmative answer
(Theorem \ref{LF}) when there is an $LF$ C*-algebra $\mathcal{D}$ such that
$\mathcal{A}\subset\mathcal{D}\subset\mathcal{A}^{\#\#}$, and $\mathcal{B}$ is
an arbitrary finite von Neumann algebra. In section $5$ we give an affirmative
answer (Theorem \ref{abelian}) when $\mathcal{A}$ is abelian (or homogeneous)
and $\mathcal{B}$ is an arbitrary von Neumann algebra.

\section{Connectedness of $\mathcal{U}_{\mathcal{B}}$ and special paths}

An \emph{internal path} in $\mathcal{U}_{\mathcal{B}}\left(  \pi\right)  ^{-}$
joining $\pi$ to $\rho$ is a continuous map $\gamma:\left[  0,1\right]
\rightarrow\mathcal{U}_{\mathcal{B}}\left(  \pi\right)  ^{-}$ such that
$\gamma\left(  0\right)  =\pi,$ $\gamma\left(  1\right)  =\rho$ and
$\gamma\left(  t\right)  \in\mathcal{U}_{\mathcal{B}}\left(  \pi\right)  $
whenever $0\leq t<1$. A \emph{strong internal path} from $\pi$ to $\rho
\in\mathcal{U}_{\mathcal{B}}\left(  \pi\right)  ^{-}$ is a continuous map
$\gamma:[0,1)\rightarrow\mathcal{U}_{\mathcal{B}}$ such that%
\[
\lim_{t\rightarrow1^{-}}\gamma\left(  t\right)  ^{\ast}\pi\left(  {}\right)
\gamma\left(  t\right)  =\rho\text{.}%
\]

In \cite[Theorem 3.9]{H} the first author proved that $\mathcal{U}%
_{\mathcal{B}}\left(  T\right)  ^{-}$ is always path connected when
$\mathcal{B}=B\left(  \ell^{2}\right)  $. Actually a slightly stronger result
was proved.

\begin{theorem}
\label{old}\cite[Theorem 3.9]{H} Suppose $X\in B\left(  \ell^{2}\right)  $ and
$Y\in\mathcal{U}_{B\left(  \ell^{2}\right)  }\left(  X\right)  ^{-}$. Then
there is a $W$ such that

\begin{enumerate}
\item $W$ is unitarily equivalent to $W\oplus W\oplus\cdots$ ,

\item $X\oplus W$ is unitarily equivalent to $Y\oplus W$,

\item If $C\in B\left(  \ell^{2}\right)  $ is unitarily equivalent to $X\oplus
W$, then

\begin{enumerate}
\item $C\in\mathcal{U}_{B\left(  \ell^{2}\right)  }\left(  X\right)
^{-}=\mathcal{U}_{B\left(  \ell^{2}\right)  }\left(  Y\right)  ^{-}$,

\item there is a strong internal path in $\mathcal{U}_{B\left(  \ell
^{2}\right)  }\left(  X\right)  ^{-}$ from $X$ to $C$, and

\item there is a strong internal path in $\mathcal{U}_{B\left(  \ell
^{2}\right)  }\left(  Y\right)  ^{-}$from $Y$ to $C$.
\end{enumerate}
\end{enumerate}
\end{theorem}

There is no reason, a priori, that $\mathcal{U}_{\mathcal{B}}\left(
\pi\right)  $ is even connected. It is well-known that if $P$ and $Q$ are
projections in a unital C*-algebra $\mathcal{B}$ and $\left\Vert
P-Q\right\Vert <1$, then $P$ and $Q$ are unitarily equivalent \cite{R}. It was
proved in \cite{DH} that two unital representations $\pi,\rho$ of a
finite-dimensional C*-algebras $\mathcal{A}$ are unitarily equivalent if and
only if $\pi\left(  p\right)  $ is unitarily equivalent to $\rho\left(
p\right)  $ for every minimal projection $p\in\mathcal{A}$.

If $\mathcal{U}_{\mathcal{B}}$ is connected, then every $\mathcal{U}%
_{\mathcal{B}}\left(  \pi\right)  $ must be connected. If $x\in\mathcal{B}$
and $\left\Vert 1-x\right\Vert <1$ , then $(-\infty,-1]\cap\sigma\left(
x\right)  =\varnothing,$  so $A\left(  x\right)  =-\log\left(  x\right)
\in\mathcal{B}$, $A\left(  x\right)  =A\left(  x\right)  ^{\ast},$ and
$x=e^{iA\left(  x\right)  }.$ Since $t\mapsto e^{i\left(  1-t\right)  A\left(
x\right)  }$ is a path in $\mathcal{U}_{\mathcal{B}}$ from $x$ to $1$, we see
that $\left\{  x\in\mathcal{U}_{\mathcal{B}}:\left\Vert 1-x\right\Vert
<1\right\}  $ is contained in the path component $W$ of $1$in $\mathcal{U}%
_{\mathcal{B}}$.  Since $W=\cup uW$ such that $u\in W$, we see that $W$ is
open in $\mathcal{U}_{\mathbb{B}}$. Thus $\mathcal{U}_{\mathcal{B}}$ is
connected if and only if it is path-connected. This means that if
$\mathcal{U}_{\mathcal{B}}$ is connected, then $\mathcal{U}_{\mathcal{B}%
}\left(  \pi\right)  $ is path-connected.

\begin{lemma}
\label{fd}If $\mathcal{A}$ is finite-dimensional, then for every $\mathcal{B}$
and every $\pi\in$ \textrm{Rep}$\left(  \mathcal{A},\mathcal{B}\right)  $,
$\mathcal{U}_{\mathcal{B}}\left(  \pi\right)  $ is closed
\end{lemma}

\begin{proof}
It follows from \cite[Theorem 2 (4)]{DH} that if $\rho\in\mathcal{U}%
_{\mathcal{B}}\left(  \pi\right)  ^{-}$, then $\rho\in\mathcal{U}%
_{\mathcal{B}}\left(  \pi\right)  $.
\end{proof}

\begin{example}
B. Blackadar \cite{B} showed that in $\mathcal{B}=\mathbb{M}_{2}\left(
C\left(  S^{3}\right)  \right)  $ there are two projections $P,Q$ that are
unitarily equivalent, but are not homotopy equivalent. Thus $\mathcal{U}%
_{\mathcal{B}}\left(  P\right)  =\mathcal{U}_{\mathcal{B}}\left(  P\right)
^{-}$ is not path-connected. This implies that $\mathcal{U}_{\mathcal{B}}$ is
not connected.
\end{example}

We say that a unital C*-algebra $\mathcal{B}$ has \textbf{property UC} if
$\mathcal{U}_{\mathcal{B}}$ is connected. The algebra $\mathcal{B}$ has
\textbf{property HUC} if, for every projection $P\in\mathcal{B}$,
$P\mathcal{B}P$ has property UC. We say that $\mathcal{B}$ is
\emph{matricially stable} if and only if, for every $n\in\mathbb{N}$,
$\mathcal{B}$ is isomorphic to $\mathbb{M}_{n}\left(  \mathcal{B}\right)  $.

\begin{lemma}
The following are true:

\begin{enumerate}
\item Every von Neumann algebra has property HUC.

\item A direct limit of unital C*-algebras with property HUC has property HUC.

\item Every unital AF algebra has property HUC.

\item If $\mathcal{A}$ is a unital C*-algebra and, for every $n\in\mathbb{N}$,
$\mathbb{M}_{n}\left(  \mathbb{C}\right)  $ has property UC, then
$K_{1}\left(  \mathcal{A}\right)  =0$.

\item If $\mathcal{B}$ is matricially stable, then $\mathcal{B}$ has property
UC if and only if $K_{1}\left(  \mathcal{B}\right)  =0$.
\end{enumerate}
\end{lemma}

\begin{proof}
$\left(  1\right)  $. In a von Neumann algebra $\mathcal{A}$ every unitary $U$
can be written $U=e^{iA}$ with $A=A^{\ast}$, and the path $g\left(  t\right)
=e^{\left(  1-t\right)  iA}$ connects $U$ to $1$ in $\mathcal{U}_{\mathcal{A}%
}$. Thus $\mathcal{A}$ has property UC. But $P\mathcal{A}P$ is a von Neumann
algebra for every projection $P\in\mathcal{A}$. Thus $\mathcal{A}$ has
property HUC.

$\left(  2\right)  $. Suppose $\left\{  \mathcal{A}_{\lambda}:\lambda
\in\Lambda\right\}  $ is an increasingly directed family of unital
C*-subalgebras of a unital C*-subalgebras $\mathcal{A}$ with property UC, and
$\mathcal{A}=\left[  \cup_{\lambda\in\Lambda}\mathcal{A}_{\lambda}\right]
^{-}.$ Let $K$ be the connected component of $\mathcal{U}_{\mathcal{A}}$ that
contains $1$. Suppose $U\in\mathcal{U}_{\mathcal{A}}$ and $\varepsilon>0$.
Then there is a $\lambda\in\Lambda$ and a unitary $V\in\mathcal{A}_{\lambda}$
such that $\left\Vert U-V\right\Vert <\varepsilon$. However, if $E_{\lambda}$
is the connected component of $\mathcal{U}_{\mathcal{A}_{\lambda}}$, we have
$V\in E_{\lambda}\subset E$. Since $\varepsilon>0$, $U\in E^{-}=E$. Next
suppose each $\mathcal{A}_{\lambda}$ has property HUC and $P\in\mathcal{A}$ is
a projection. Then there is a $\lambda_{0}\in\Lambda$ and a projection
$Q\in\mathcal{A}_{\lambda_{0}}$ such that $\left\Vert P-Q\right\Vert <1$,
which implies there is a unitary $W\in\mathcal{A}$ such that $P=W^{\ast}QW$.
Hence%
\[
P\mathcal{A}P=W^{\ast}QW\mathcal{A}W^{\ast}QW=W^{\ast}\left(  Q\mathcal{A}%
Q\right)  .
\]
Thus $P\mathcal{A}P$ is isomorphic to
\[
Q\mathcal{A}Q=\left[  \cup_{\lambda\geq\lambda_{0}}Q\mathcal{A}_{\lambda
}Q\right]  ^{-}.
\]
Since each $Q\mathcal{A}_{\lambda}Q$ has property UC when $Q\in\mathcal{A}%
_{\lambda}$, we see that $P\mathcal{A}P$ has property UC. Thus $\mathcal{A}$
has property HUC.

$\left(  3\right)  $. This follows from $\left(  1\right)  $ and $\left(
2\right)  $.

$\left(  4\right)  $. This follows from the definition of $K_{1}\left(
\mathcal{A}\right)  $.

$\left(  5\right)  $. This follows from $\left(  4\right)  $.
\end{proof}

\section{$B\left(  \ell^{2}\right)  $}

In this section we extend Theorem \ref{old} to the case where the single
operator is replaced with a representation of a separable C*-algebra. The key
idea is a result of C. Olsen and W. Zame \cite{OZ} that if $\mathcal{A}$ is a
separable C*-algebra, then $\mathcal{A}\otimes\mathcal{K}\left(  \ell
^{2}\right)  $ is singly generated. This gives us a general technique for
relating the separable case to the singly generated case.

Suppose $\mathcal{A}$ is a unital C*-algebra. Let $\mathcal{A}^{\dagger}$
denote the unitization of $\mathcal{A}\otimes\mathcal{K}\left(  \ell
^{2}\right)  $. If $\pi\in$ \textrm{Rep}$\left(  \mathcal{A},\mathcal{B}%
\right)  $ we define $\pi^{\dagger}:\mathcal{A}^{\dagger}\rightarrow
\mathcal{B}^{\dagger}$ by%
\[
\pi^{\dagger}\left(  \lambda1+\left(  a_{ij}\right)  \right)  =\lambda
1+\left(  \pi\left(  a_{ij}\right)  \right)  \text{.}%
\]
Let $\mathcal{B}^{\maltese}$ be the C*-algebra generated by $\mathcal{B}%
^{\dagger}$ and $\left\{  diag\left(  a,a,\ldots\right)  :a\in\mathcal{A}%
\right\}  $.

\begin{theorem}
\label{tensor}Suppose $\mathcal{A}$ and $\mathcal{B}$ are unital C*-algebras
and $\pi,\rho\in$ \textrm{Rep}$\left(  \mathcal{A},\mathcal{B}\right)  $. Then

\begin{enumerate}
\item The map $\rho\mapsto\rho^{\dagger}$ from \textrm{Rep}$\left(
\mathcal{A},\mathcal{B}\right)  $ to \textrm{Rep}$\left(  \mathcal{A}%
^{\dagger},\mathcal{B}^{\dagger}\right)  $ is continuous.

\item If $\pi,\rho\in$ \textrm{Rep}$\left(  \mathcal{A},\mathcal{B}\right)  $,
then%
\[
\rho\in\mathcal{U}_{\mathcal{B}}\left(  \pi\right)  ^{-}\text{ if and only if
}\rho^{\dagger}\in\mathcal{U}_{\mathcal{B}^{\dagger}}\left(  \pi^{\dagger
}\right)  ^{-}\text{. }%
\]

\item If $\rho$ $\in\mathcal{U}_{\mathcal{B}}\left(  \pi\right)  ^{-}$ and
there is an internal path in $\mathcal{U}\left(  \pi\right)  ^{-}$ joining
$\pi$ to $\rho$, then there is an internal path in $\mathcal{U}_{\mathcal{B}%
^{\maltese}}\left(  \pi^{\dagger}\right)  ^{-}$ joining $\pi^{\mathcal{\dagger
}}$ to $\rho^{\mathcal{\dagger}}$.

\item If

\begin{enumerate}
\item $\mathcal{B}^{\dagger}\subset\mathcal{E}$ is a C*-algebra with
$e_{11}\mathcal{E}e_{11}=e_{11}\mathcal{B}^{\dagger}e_{11}$,

\item $\rho_{1}\in\mathcal{U}_{\mathcal{E}}\left(  \pi^{\dagger}\right)  ^{-}$,

\item For every $a\in\mathcal{A}$,
\[
\rho_{1}\left(  diag\left(  a,0,0,\ldots\right)  \right)  =diag\left(
\rho\left(  a\right)  ,0,0,\ldots\right)
\]

\item $\mathcal{U}_{\mathcal{B}}$ is connected,

\item $\mathcal{B}^{\dagger}\subset\mathcal{E}$ is a C*-algebra with
$e_{11}\mathcal{E}e_{11}=e_{11}\mathcal{B}^{\dagger}e_{11},$ and

\item there is a strong internal path in $\mathcal{U}_{\mathcal{E}}\left(
\pi^{\dagger}\right)  ^{-}$ from $\pi^{\dagger}$ to $\rho_{1}$,
\end{enumerate}

then there is a strong internal path in $\mathcal{U}_{\mathcal{B}}\left(
\pi\right)  ^{-}$ from $\pi$ to $\rho$.
\end{enumerate}
\end{theorem}

\begin{proof}
$\left(  1\right)  $. This is obvious.

$\left(  2\right)  $. Suppose $\rho\in\mathcal{U}_{\mathcal{B}}\left(
\pi\right)  ^{-}$. Then there is a sequence $\left\{  U_{n}\right\}  $ in
$\mathcal{U}_{\mathcal{B}}$ such that, for every $a\in\mathcal{A}$,%
\[
\lim_{n\rightarrow\infty}\left\Vert U_{n}\pi\left(  a\right)  U_{n}^{\ast
}-\rho\left(  a\right)  \right\Vert =0\text{.}%
\]
For each positive integer $n,$ let $W_{n}=diag\left(  U_{n},\ldots
,U_{n},1,1,1,\ldots\right)  $ in $\mathcal{B}^{\dagger}$ (with $U_{n}$
repeated $n$ times). Since
\[
\left\{  T\in\mathcal{A}^{\dagger}:\lim_{n\rightarrow\infty}\left\Vert
W_{n}\pi\left(  a\right)  W_{n}^{\ast}-\rho\left(  a\right)  \right\Vert
=0\right\}
\]
is a unital subalgebra containing the operators $\left(  A_{ij}\right)
\in\mathcal{A}^{\dagger}$ such that,
\[
\left\{  \left(  i,j\right)  \in\mathbb{N}\times\mathbb{N}:\left(  i,j\right)
\neq\left(  0,0\right)  \right\}
\]
is finite, we see that $\rho^{\dagger}\in\mathcal{U}_{\mathcal{B}^{\dagger}%
}\left(  \pi^{\dagger}\right)  ^{-}$.

Conversely, suppose $\rho^{\dagger}\in\mathcal{U}_{\mathcal{B}^{\dagger}%
}\left(  \pi^{\dagger}\right)  ^{-}$. Then there is a sequence $\left\{
V_{n}\right\}  $ in $\mathcal{B}^{\dagger}$ such that, for every
$T\in\mathcal{A}^{\dagger}$,%
\[
\lim_{n\rightarrow\infty}\left\Vert V_{n}\pi^{\dagger}\left(  T\right)
V_{n}^{\ast}-\rho^{\dagger}\left(  T\right)  \right\Vert =0\text{.}%
\]
Since $\pi^{\dagger}\left(  e_{11}\right)  =\rho^{\dagger}\left(
e_{11}\right)  =e_{11}$, we see that%
\[
\lim_{n\rightarrow\infty}\left\Vert V_{n}e_{11}-e_{11}V_{n}\right\Vert
=\lim_{n\rightarrow\infty}\left\Vert V_{n}\pi^{\dagger}\left(  e_{11}\right)
V_{n}^{\ast}-\rho^{\dagger}\left(  e_{11}\right)  \right\Vert =0\text{,}%
\]
and there is a sequence $\left\{  W_{n}\right\}  $ in $\mathcal{U}%
_{\mathcal{A}}$ such that%
\[
\lim_{n\rightarrow\infty}\left\Vert W_{n}-e_{11}V_{n}e_{11}%
|_{\text{\textrm{ran}}\left(  e_{11}\right)  }\right\Vert =0.
\]
Hence, for every $a\in\mathcal{A}$,%
\[
\lim_{n\rightarrow\infty}\left\Vert W_{n}\pi\left(  a\right)  W_{n}^{\ast
}-\rho\left(  a\right)  \right\Vert =0\text{.}%
\]
Thus $\rho\in\mathcal{U}_{\mathcal{B}}\left(  \pi\right)  ^{-}$.

$\left(  3\right)  $. Suppose there is an internal path $\gamma:\left[
0,1\right]  \rightarrow\mathcal{U}\left(  \pi\right)  ^{-}$ joining $\pi$ to
$\rho$. For $0\leq t<1$ write $\gamma\left(  t\right)  =U_{t}\pi\left(
{}\right)  U_{t}^{\ast}$ with $U_{t}\in\mathcal{U}_{\mathcal{B}}$. For each
$0\leq t<1$ let $V_{t}=diag\left(  U_{t},U_{t},\ldots\right)  \in
\mathcal{U}_{\mathcal{B}^{\maltese}}$ and let $\Gamma\left(  t\right)
=V_{t}\pi^{\mathcal{\dagger}}\left(  {}\right)  V_{t}^{\ast}$. Then, for every
$T\in\mathcal{A}^{\dagger}$,%
\[
\lim_{t\rightarrow1^{-}}\left\Vert V_{t}\pi^{\mathcal{\dagger}}\left(
T\right)  V_{t}^{\ast}-\rho^{\mathcal{\dagger}}\left(  T\right)  \right\Vert
=0\text{.}%
\]

$\left(  4\right)  $. Suppose $\Gamma:[0,1)\rightarrow\mathcal{U}%
_{\mathcal{E}}$ is continuous, and, for every $T\in\mathcal{A}^{\dagger}$,%
\[
\lim_{t\rightarrow1^{-}}\left\Vert \Gamma\left(  t\right)  \pi
^{\mathcal{\dagger}}\left(  T\right)  \Gamma\left(  t\right)  ^{\ast}-\rho
_{1}\left(  T\right)  \right\Vert =0\text{.}%
\]
Since $\rho_{1}\left(  e_{11}\right)  =\pi^{\mathcal{\dagger}}\left(
e_{11}\right)  =e_{11}$, we conclude that
\[
\lim_{t\rightarrow1^{-}}\left\Vert \Gamma\left(  t\right)  e_{11}-e_{11}%
\Gamma\left(  t\right)  \right\Vert =
\]%
\[
\lim_{t\rightarrow1^{-}}\left\Vert \Gamma\left(  t\right)  \pi
^{\mathcal{\dagger}}\left(  e_{11}\right)  \Gamma\left(  t\right)  ^{\ast
}-\rho_{1}\left(  e_{11}\right)  \right\Vert =0\text{.}%
\]
Since $\Gamma\left(  t\right)  $ is unitary, there is a $t_{0}\in\lbrack0,1)$
such that, whenever $t_{0}\leq t<1$ , we have $C_{t}=e_{11}\Gamma\left(
t\right)  e_{11}$ is invertible in $\mathcal{A}$ and if
\[
U_{t}=C_{t}\left[  C_{t}^{\ast}C_{t}\right]  ^{-1/2},
\]
then $U_{t}\in\mathcal{U}_{\mathcal{A}}$ and
\[
\lim_{t\rightarrow1^{-}}\left\Vert C_{t}-U_{t}\right\Vert =0.
\]
Since $\mathcal{U}_{\mathcal{A}}$ is connected, there is a continuous map
$t\mapsto U_{t}\in\mathcal{U}_{\mathcal{A}}$ for $0\leq t\leq t_{0}$ so that
$U_{0}=1.$ If, for every $a\in\mathcal{A}$, we consider $T_{a}=diag\left(
a,0,0,\ldots\right)  $, it is easily seen that%
\[
\lim_{t\rightarrow1^{-}}\left\Vert U_{t}\pi\left(  a\right)  U_{t}^{\ast}%
-\rho\left(  a\right)  \right\Vert =0.
\]

\end{proof}

\bigskip

\begin{theorem}
\label{separable}Suppose $\mathcal{A}$ is a separable unital C*-algebra and
$\pi\in$ \textrm{Rep}$\left(  \mathcal{A},B\left(  \ell^{2}\right)  \right)
$. Then $\mathcal{U}_{B\left(  \ell^{2}\right)  }\left(  \pi\right)  ^{-}$ is path-connected.
\end{theorem}

\begin{proof}
Suppose $\rho\in\mathcal{U}_{B\left(  \ell^{2}\right)  }\left(  \pi\right)
^{-}$. Then, by Theorem \ref{tensor}, $\rho^{\dagger}\in\mathcal{U}_{B\left(
\ell^{2}\right)  ^{\dagger}}\left(  \pi^{\dagger}\right)  ^{-}.$ But $B\left(
\ell^{2}\right)  ^{\dagger}\subset B\left(  \ell^{2}\oplus\ell^{2}\oplus
\cdots\right)  =\mathcal{E}$. Also, by \cite{OZ} there is an operator
$T\in\mathcal{A}^{\dagger}$ such that $\mathcal{A}^{\dagger}=C^{\ast}\left(
T\right)  $. Thus $\rho\left(  T\right)  \in\mathcal{U}_{\mathcal{E}}\left(
\pi\left(  T\right)  \right)  ^{-}$. We know from Theorem \ref{old} with
$X=\pi^{\dagger}\left(  T\right)  $ and $Y=\rho^{\dagger}\left(  T\right)  $,
that there is a $W$ in $\mathcal{E}$ such $C\in\mathcal{U}_{\mathcal{E}%
}\left(  \pi\left(  T\right)  \right)  ^{-}$ and a strong internal path from
$\pi\left(  T\right)  $ to $C$ in $\mathcal{U}_{\mathcal{E}}\left(  \pi\left(
T\right)  \right)  ^{-}$ and a strong internal path in $\mathcal{U}%
_{\mathcal{E}}\left(  \rho\left(  T\right)  \right)  ^{-}$ from $\rho\left(
T\right)  $ to $C$. There is a representation $\delta$ of $C^{\ast}\left(
T\right)  $ such that $\delta_{0}\left(  T\right)  =W$ and if $\delta\left(
A\right)  =A\oplus\delta\left(  A\right)  $, we have $\delta\left(  T\right)
=T\oplus W$. Since $e_{11}$ and $\delta\left(  e_{11}\right)  =e_{11}%
\oplus\delta_{0}\left(  e_{11}\right)  $ are projections with infinite rank
and infinite corank, there is a unitary operator $V$ such that $V^{\ast}%
\delta\left(  e_{11}\right)  V=e_{11}$ and $V^{\ast}TV\in\mathcal{E}$. Let
$C=V^{\ast}\delta\left(  T\right)  V$ and $\rho_{1}\left(  {}\right)
=V^{\ast}\delta\left(  {}\right)  V$. It follows that there is a $\sigma\in$
\textrm{Rep}$\left(  \mathcal{A},B\left(  \ell^{2}\right)  \right)  $ such
that, for every $a\in\mathcal{A}$,%
\[
\rho_{1}\left(  diag\left(  a,0,0,\cdots\right)  \right)  =\left(
\sigma\left(  a\right)  ,0,0,\cdots\right)  \text{.}%
\]
Since there is an internal path in $\mathcal{U}_{\mathcal{E}}\left(
\pi^{\dagger}\left(  T\right)  \right)  ^{-}$ from $\pi^{\dagger}\left(
T\right)  $ to $\rho_{1}\left(  T\right)  $, there is a strong internal path
in $\mathcal{U}_{\mathcal{E}}\left(  \pi^{\dagger}\right)  ^{-}$ from
$\pi^{\dagger}$ to $\rho_{1}$. It follows from part $\left(  4\right)  $ of
Theorem \ref{tensor} that there is a strong internal path in $\mathcal{U}%
_{B\left(  \ell^{2}\right)  }\left(  \pi\right)  ^{-}$ from $\pi$ to $\sigma.$
Similarly, there is a strong internal path in $\mathcal{U}_{B\left(  \ell
^{2}\right)  }\left(  \rho\right)  ^{-}$ from $\rho$ to $\sigma$. Thus there
is a path in $\mathcal{U}_{B\left(  \ell^{2}\right)  }\left(  \pi\right)
^{-}=\mathcal{U}_{B\left(  \ell^{2}\right)  }\left(  \rho\right)  ^{-}$ from
$\pi$ to $\rho$.
\end{proof}

\bigskip

\section{AF algebras\bigskip}

\begin{lemma}
\label{stitch}Suppose $1\in\mathcal{A}\subset\mathcal{D}$ are separable unital
C*-algebras, $\mathcal{B}$ is a unital C*-algebra and $\pi,\rho\in$
\textrm{Rep}$\left(  \mathcal{D},\mathcal{B}\right)  $, and suppose
$V,W\in\mathcal{U}_{\mathcal{B}}$ such that

\begin{enumerate}
\item for every $x\in\mathcal{D}$,%
\[
W^{\ast}\rho\left(  x\right)  W=\pi\left(  x\right)  \text{,}%
\]

\item for every $x\in\mathcal{A}$,
\[
V^{\ast}\rho\left(  x\right)  V=\pi\left(  x\right)  \text{,}%
\]

\item $\mathcal{U}_{\mathcal{B}\cap\rho\left(  \mathcal{A}\right)  ^{\prime}}$
is connected.
\end{enumerate}

Then there is a path $t\mapsto U_{t}$ of unitary operators in $\mathcal{B}$
such that $U_{0}=V$, $U_{1}=W$, and for every $t\in\left[  0,1\right]  $ and
every $x\in\mathcal{A}$,%
\[
U_{t}^{\ast}\rho\left(  x\right)  U_{t}=\pi\left(  x\right)  \text{.}%
\]

\end{lemma}

\begin{proof}
We know that, for every $x\in\mathcal{A}$,%
\[
W^{\ast}\rho\left(  x\right)  W=V^{\ast}\rho\left(  x\right)  V.
\]
Thus $VW^{\ast}=X\in\rho\left(  \mathcal{A}\right)  ^{\prime}\cap\mathcal{B}$.
Thus $W=U^{\ast}V$. Since $\mathcal{U}_{\rho\left(  \mathcal{A}\right)
^{\prime}\cap\mathcal{B}}$ is path connected, there is a path $t\mapsto X_{t}$
of unitary elements in$\ \rho\left(  \mathcal{A}\right)  ^{\prime}%
\cap\mathcal{B}$ such that $X_{0}=1$ and $X_{1}=X$. For $t\in\left[
0,1\right]  $ let $U_{t}=X_{t}^{\ast}V$. Then $U_{t}$ is a path in
$\mathcal{U}_{\mathcal{B}}$, $U_{0}=V$ and $U_{1}=X^{\ast}V=W$. Moreover, for
each $t\in\left[  0,1\right]  $ and each $x\in\mathcal{A}$,
\[
U_{t}^{\ast}\rho\left(  x\right)  U_{t}=V^{\ast}X_{t}\rho\left(  x\right)
X_{t}^{\ast}V=V^{\ast}\rho\left(  x\right)  V=\pi\left(  x\right)  \text{.}%
\]

\end{proof}

\bigskip

\begin{theorem}
\label{direct limit}Suppose $\mathcal{A}_{1}\subset\mathcal{A}_{2}%
\subset\cdots\subset\mathcal{A}$ and $\mathcal{A}=\left[  \cup_{n\in
\mathbb{N}}\mathcal{A}_{n}\right]  ^{-}$ is separable. Suppose $\pi,\rho\in$
\textrm{Rep}$\left(  \mathcal{A},\mathcal{B}\right)  $ such that, for every
$n\in\mathbb{N}$,

\begin{enumerate}
\item $\rho|_{\mathcal{A}_{n}}\in\mathcal{U}_{\mathcal{B}}\left(
\pi|_{\mathcal{A}_{n}}\right)  $,

\item $\mathcal{U}_{\rho\left(  \mathcal{A}_{n}\right)  ^{\prime}%
\cap\mathcal{B}}$ is connected.
\end{enumerate}

Then there is a strong internal path from $\pi$ to $\rho$.
\end{theorem}

\begin{proof}
For each $n\in\mathbb{N}$, choose $U_{n}\in\mathcal{U}_{\mathcal{B}}$ such
that, for every $a\in\mathcal{A}_{n}$,%
\[
U_{n}^{\ast}\rho\left(  a\right)  U_{n}=\pi\left(  a\right)  \text{.}%
\]
It follows from Lemma \ref{stitch} that we can define a path $t\mapsto U_{t}$
from $\left[  n,n+1\right]  $ so that for $n\leq t\leq n+1$ and $a\in
\mathcal{A}_{n}$, we have
\[
U_{t}^{\ast}\rho\left(  a\right)  U_{t}=\pi\left(  a\right)  \text{.}%
\]
Thus the map $t\mapsto U_{t}$ is continuous, and, for every $a\in\cup
_{n\in\mathbb{N}}\mathcal{A}_{n}$ we have%
\[
\lim_{t\rightarrow+\infty}\left\Vert U_{t}^{\ast}\rho\left(  a\right)
U_{t}-\pi\left(  a\right)  \right\Vert =0.
\]
Hence, if we define $\pi_{t}\left(  \cdot\right)  =U_{t}^{\ast}\rho\left(
\cdot\right)  U_{t}$ for $t\in\lbrack0,\infty)$ and $\pi_{\infty}=\rho$, we
have a strong internal path in $\mathcal{U}_{\mathcal{B}}\left(  \pi\right)
^{-}$ from $\pi$ to $\rho$.
\end{proof}

\bigskip

\begin{theorem}
\label{AF}Suppose $\mathcal{A}$ is a separable unital AF C*-algebra,
$\mathcal{B}$ is a C*-algebra with property HUC, and $\pi\in$ \textrm{Rep}%
$\left(  \mathcal{A},\mathcal{B}\right)  $. Then $\mathcal{U}_{\mathcal{B}%
}\left(  \pi\right)  ^{-}$ is path-connected.
\end{theorem}

\begin{proof}
We can assume that $\ker\pi=0,$ since $\mathcal{A}/\ker\rho$ is a separable
unital AF algebra. Since $\mathcal{A}$ is unital and AF, there is a sequence
$\left\{  \mathcal{A}_{n}\right\}  $ of unital finite-dimensional
C*-subalgebras
\[
1\in\mathcal{A}_{1}\subset\mathcal{A}_{2}\subset\cdots
\]
such that%
\[
\left[  \bigcup_{n=1}^{\infty}\mathcal{A}_{n}\right]  ^{-}=\mathcal{A}\text{.}%
\]
Suppose $\rho\in\mathcal{U}_{\mathcal{M}}\left(  \pi\right)  ^{-}$. Since each
$\mathcal{A}_{n}$ is finite-dimensional, where approximate equivalence is the
same as unitary equivalence, we have $\rho|_{\mathcal{A}_{n}}\in
\mathcal{U}_{\mathcal{B}}\left(  \pi|_{\mathcal{A}_{n}}\right)  $ for each
$n\in\mathbb{N}$.

Fix $n\in\mathbb{N}$ and write $\mathcal{A}_{n}$ as $\mathbb{M}_{s_{1}}\left(
\mathbb{C}\right)  \oplus\cdots\oplus\mathbb{M}_{s_{t}}\left(  \mathbb{C}%
\right)  $ and, for $1\leq k\leq t$, let $\left\{  e_{ij,k}:1\leq i,j\leq
s_{k}\right\}  $ be the system of matrix units for $\mathbb{M}_{s_{k}}\left(
\mathbb{C}\right)  .$ It is easily seen that $\rho\left(  \mathcal{A}%
_{n}\right)  ^{\prime}\cap\mathcal{B}$ is the set of all%
\[
\sum_{k=1}^{t}\sum_{j=1}^{s_{k}}\rho\left(  e_{j1},k\right)  \rho\left(
e_{11,k}\right)  x\rho\left(  e_{11,k}\right)  \rho\left(  e_{ij,k}\right)
\]
for $x\in\mathcal{B}$. It follows that $\rho\left(  \mathcal{A}_{n}\right)
^{\prime}\cap\mathcal{B}$ is isomorphic to
\[
\sum_{1\leq k\leq t}^{\oplus}\rho\left(  e_{11,k}\right)  \mathcal{B}%
\rho\left(  e_{11,k}\right)  .
\]
Since $\mathcal{B}$ has property HUC, we see that $\rho\left(  \mathcal{A}%
_{n}\right)  ^{\prime}\cap\mathcal{B}$ has property UC. The desired conclusion
now follows from Theorem \ref{direct limit}.
\end{proof}

\bigskip

\begin{corollary}
If $\mathcal{A}$ is a separable unital AF C*-algebra and $\mathcal{B}$ is
either an $AF$ C*-algebra or a von Neumann algebra, then, for every $\rho\in$
\textrm{Rep}$\left(  \mathcal{A},\mathcal{B}\right)  $, $\mathcal{U}%
_{\mathcal{B}}\left(  \rho\right)  ^{-}$ is path-connected.
\end{corollary}

A separable C*-algebra is \emph{homogeneous} if it is a finite direct sum of
algebras of the form $\mathbb{M}_{n}\left(  C\left(  X\right)  \right)  ,$
where $X$ is a compact metric space. A unital C*-algebra is
\emph{subhomogeneous} if it is a unital subalgebra of a homogeneous
C*-algebra. Every subhomogeneous von Neumann algebra is homogeneous; in
particular, if $\mathcal{A}$ is subhomogeneous, then the second dual
$\mathcal{A}^{\#\#}$ of $\mathcal{A}$ is homogeneous. A C*-algebra is
\emph{approximately subhomogeneous} (ASH) if it is a direct limit of
subhomogeneous C*-algebras.

A (possibly nonseparable) C*-algebra $\mathcal{B}$ is LF if, for every finite
subset $F\subset\mathcal{B}$ and every $\varepsilon>0$ there is a
finite-dimensional C*-algebra $\mathcal{D}$ of $\mathcal{B}$ such that, for
every $b\in F$, \textrm{dist}$\left(  b,\mathcal{D}\right)  <\varepsilon$.
Every separable unital C*-subalgebra of a LF C*-algebra is contained in a
separable AF subalgebra. See \cite{D} for details.

We are interested in a more general property. We say that a unital C*-algebra
$\mathcal{A}$ is \emph{strongly LF-embeddable} if there is an LF C*-algebra
$\mathcal{D}$ such that $\mathcal{A}\subset\mathcal{D}\subset\mathcal{A}%
^{\#\#}$. It is easily shown that an ASH algebra is strongly LF-embeddable,
i.e., if $\left\{  \mathcal{A}_{\lambda}\right\}  $ is an increasingly
directed family of subhomogeneous C*-algebras and $\mathcal{A}=\left(
\cup_{\lambda}\mathcal{A}_{\lambda}\right)  ^{-\left\Vert {}\right\Vert }$,
then $\mathcal{A}\subset\left(  \cup_{\lambda}\mathcal{A}_{\lambda}%
^{\#\#}\right)  ^{-\left\Vert {}\right\Vert }\subset\mathcal{A}^{\#\#}$. The
proof of the next theorem relies on results in \cite{HL}.

\begin{theorem}
\label{LF}Suppose $\mathcal{A}$ is a separable strongly LF embeddable
C*-algebra and $\mathcal{M}$ is a finite von Neumann algebra. Then, for every
$\pi\in$ \textrm{Rep}$\left(  \mathcal{A},\mathcal{M}\right)  $,
$\mathcal{U}_{\mathcal{M}}\left(  \pi\right)  ^{-}$ is path connected.
\end{theorem}

\begin{proof}
Suppose $\rho\in\mathcal{U}_{\mathcal{M}}\left(  \pi\right)  ^{-}$. It follows
that there are weak*-weak* continuous unital $\ast$-homomorphisms $\hat{\pi
},\hat{\rho}:\mathcal{A}^{\#\#}\rightarrow\mathcal{M}$ such that $\hat{\pi
}|_{\mathcal{A}}=\pi$ and $\hat{\rho}|_{\mathcal{A}}=\rho$. Since
$\mathcal{A}$ is strongly LF embeddable, there is a separable unital AF
C*-algebra $\mathcal{D}$ such that
\[
\mathcal{A}\subset\mathcal{D}\subset\mathcal{A}^{\#\#}\text{.}%
\]
It follows from \cite[Theorem 2]{HL} that $\hat{\rho}|_{\mathcal{D}}%
\in\mathcal{U}_{\mathcal{M}}\left(  \hat{\pi}|_{\mathcal{D}}\right)  ^{-}$. We
know from Theorem \ref{AF} that $\mathcal{U}_{\mathcal{M}}\left(  \hat{\pi
}|_{\mathcal{D}}\right)  ^{-}$ is path connected. Thus there is a path in
$\mathcal{U}_{\mathcal{M}}\left(  \hat{\pi}|_{\mathcal{D}}\right)  ^{-}$ from
$\hat{\pi}|_{\mathcal{D}}$ to $\hat{\rho}|_{\mathcal{D}}$. Restricting to
$\mathcal{A}$, we obtain a path in $\mathcal{U}_{\mathcal{M}}\left(
\pi\right)  ^{-}$ from $\pi$ to $\rho$.
\end{proof}

\section{Abelian algebras}

\begin{lemma}
\label{sep}Suppose $\mathcal{N}$ is a countably generated von Neumann algebra.
Then $\mathcal{N}$ is isomorphic to a direct sum $\sum_{i\in I}^{\oplus
}\mathcal{N}_{i}$ so that each $\mathcal{N}_{i}$ acts on a separable Hilbert
space. In particular, each $\mathcal{N}_{i}$ is $\sigma$-finite.
\end{lemma}

\begin{proof}
Suppose $N\subset B\left(  H\right)  $ and $e\in H$ with $\left\Vert
e\right\Vert =1.$ Let $P$ be the orthogonal projection onto $\left(
\mathcal{N}e\right)  ^{-}$. Thus $P\in\mathcal{N}^{\prime}$. Let $P_{e}%
\in\mathcal{Z}\left(  \mathcal{N}\right)  $ be the central cover of $P.$ Then
the map $T\mapsto T|_{P\left(  H\right)  }$ is a normal isomorphism between
$N|_{P_{e}\left(  H\right)  }$ and $N|_{P\left(  H\right)  }.$ Since
$\mathcal{N}$ is countably generated, $P\left(  H\right)  $ is separable. Thus
$N|_{P_{e}\left(  H\right)  }$ is a direct summand of $\mathcal{N}$ that is
isomorphic to a von Neumann algebra on a separable Hilbert space. The rest of
the proof follows from this idea and Zorn's lemma.
\end{proof}

\bigskip

Suppose $\mathcal{M}$ is a von Neumann algebra and $T\in\mathcal{M}$. In
\cite{DH} H. Ding and D. Hadwin defined $\mathcal{M}$-rank$\left(  T\right)  $
to be the Murray von Neumann equivalence class of the orthogonal projection
$\mathfrak{R}\left(  T\right)  $ onto the closure of the range of $T$. We say
$\mathcal{M}$-rank$\left(  S\right)  \leq\mathcal{M}$-rank$\left(  T\right)  $
if and only if there is a projection $P\in\mathcal{M}$ such that
$P\leq\mathfrak{R}\left(  T\right)  $ and $P$ is Murray von Neumann equivalent
to $\mathfrak{R}\left(  S\right)  $. They proved that if a separable unital
C*-algebra is a direct limit of homogeneous algebras, and $\mathcal{M}$ acts
on a separable Hilbert space, then for all $\pi,\rho\in$ \textrm{Rep}$\left(
\mathcal{A},\mathcal{M}\right)  $, $\rho\in\mathcal{U}_{\mathcal{M}}\left(
\pi\right)  ^{-}$ if and only if, for every $x\in\mathcal{A}$,%
\[
\mathcal{M}\text{-rank}\left(  \pi\left(  x\right)  \right)  =\mathcal{M}%
\text{-rank}\left(  \rho\left(  x\right)  \right)  \text{.}%
\]
A key ingredient of the proof of this result was a sequential semicontinuity
of $\mathcal{M}$-rank with respect to the *-SOT that was proved when
$\mathcal{M}$ is a von Neumann algebra acting on a separable Hilbert space
\cite[Theorem 1]{DH}. We extend this to the general case.

\begin{lemma}
\label{semicont}Suppose $\mathcal{M}$ is a von Neumann algebra, $A,B\in
\mathcal{M}$ and, for each $n\in\mathbb{N}$, $B_{n}\in\mathcal{M}$ and
$\mathcal{M}$-rank$\left(  B_{n}\right)  \leq\mathcal{M}$-rank$\left(
A\right)  $. If $B_{n}\rightarrow B$ is the $\ast$-SOT, then $\mathcal{M}%
$-rank$\left(  B\right)  \leq\mathcal{M}$-rank$\left(  A\right)  $.
\end{lemma}

\begin{proof}
Let $P_{n}=\mathfrak{R}\left(  B_{n}\right)  ,$ $Q=\mathfrak{R}\left(
A\right)  ,$ and, for each $n\in\mathbb{N}$, choose a partial isometry
$V_{n}\in\mathcal{M}$ such that $V_{n}^{\ast}V_{n}=P_{n}$ and $V_{n}%
V_{n}^{\ast}\leq Q$. Let
\[
\mathcal{N}=W^{\ast}\left(  \left\{  A,B,B_{1},V_{1},B_{2},V_{2}%
,\ldots\right\}  \right)  .
\]
Clearly, we have, for every $n\in\mathbb{N}$, that
\[
\mathcal{N}\text{-rank}\left(  B_{n}\right)  \leq\mathcal{N}\text{-rank}%
\left(  A\right)  .
\]
By Lemma \ref{sep}, we can write%
\[
\mathcal{N}=\sum_{i\in I}^{\oplus}\mathcal{N}_{i}%
\]
with each $\mathcal{N}_{i}$ acting on a separable Hilbert space.

Write%
\[
A=\sum_{i\in I}^{\oplus}A_{i},B=\sum_{i\in I}^{\oplus}B_{i},B_{n}=\sum_{i\in
I}^{\oplus}B_{n,i},V_{n}=\sum_{i\in I}^{\oplus}V_{n,i}.
\]
Since $\mathfrak{R}\left(  A\right)  =\sum_{i\in I}^{\oplus}\mathfrak{R}%
\left(  A_{i}\right)  $ and $\mathfrak{R}\left(  B\right)  =\sum_{i\in
I}^{\oplus}\mathfrak{R}\left(  B_{n,i}\right)  $ , for each $i\in I$,
$\mathcal{N}_{i}$-rank$\left(  B_{n,i}\right)  \leq\mathcal{N}_{i}%
$-rank$\left(  A_{i}\right)  $ and the limit in the $\ast$-$SOT$ of $B_{n,i}$
is $B_{i}$. Thus, by \cite[Theorem 1]{DH}, for each $i\in I$,%
\[
\mathcal{N}_{i}\text{-rank}\left(  B_{i}\right)  \leq\mathcal{N}%
_{i}\text{-rank}\left(  A_{i}\right)  \text{.}%
\]
Thus, for each $i\in I$, there is a partial isometry $W_{i}\in\mathcal{N}_{i}$
such that%
\[
W_{i}^{\ast}W_{i}=\mathfrak{R}\left(  B_{i}\right)  \text{ and }W_{i}%
W_{i}^{\ast}\leq\mathfrak{R}\left(  A_{i}\right)  .
\]
Then $W=\sum_{i\in I}^{\oplus}W_{i}$ is a partial isometry in $\mathcal{N}$
such that%
\[
W^{\ast}W=\mathfrak{R}\left(  B\right)  \text{ and }WW^{\ast}\leq
\mathfrak{R}\left(  A\right)  .
\]
Since we also have $W\in\mathcal{M}$, we conclude $\mathcal{M}$-rank$\left(
B\right)  \leq\mathcal{M}$-rank$\left(  A\right)  $.
\end{proof}

\bigskip

\begin{corollary}
\label{rank}If $\mathcal{A}$ is a unital C*-algebra, $\mathcal{M}$ is a von
Neumann algebra and $\pi\in$ \textrm{Rep}$\left(  \mathcal{A},\mathcal{M}%
\right)  $ and $\rho\in\mathcal{U}_{\mathcal{M}}\left(  \pi\right)  ^{-}$,
then, for every $a\in\mathcal{A}$,%
\[
\mathcal{M}\text{-rank}\left(  \pi\left(  a\right)  \right)  =\mathcal{M}%
\text{-rank}\left(  \rho\left(  a\right)  \right)  \text{.}%
\]

\end{corollary}

\begin{proof}
Suppose $a\in\mathcal{A}$. There is a sequence $\left\{  U_{n}\right\}  $ in
$\mathcal{U}_{\mathcal{M}}$ such that%
\[
\lim_{n\rightarrow\infty}\left\Vert U_{n}^{\ast}\pi\left(  A\right)
U_{n}-\rho\left(  A\right)  \right\Vert =\lim_{n\rightarrow\infty}\left\Vert
\pi\left(  a\right)  -U_{n}\rho\left(  a\right)  U_{n}^{\ast}\right\Vert =0.
\]
Also $\mathcal{M}$-rank$\left(  U_{n}^{\ast}\pi\left(  a\right)  U_{n}\right)
=\mathcal{M}$-rank$\left(  \pi\left(  a\right)  \right)  $ and $\mathcal{M}%
$-rank$\left(  U_{n}\rho\left(  a\right)  U_{n}^{\ast}\right)  =\mathcal{M}%
$-rank$\left(  \rho\left(  a\right)  \right)  $ for each $n\in\mathbb{N}$.
Thus, by Lemma \ref{semicont},%
\[
\mathcal{M}\text{-rank}\left(  \rho\left(  a\right)  \right)  \leq
\mathcal{M}\text{-rank}\left(  \pi\left(  a\right)  \right)  \text{ and
}\mathcal{M}\text{-rank}\left(  \pi\left(  a\right)  \right)  \leq
\mathcal{M}\text{-rank}\left(  \rho\left(  a\right)  \right)  \text{.}%
\]

\end{proof}

Suppose $\mathcal{A}$ is a unital C*-algebra and $\mathcal{M}$ is a von
Neumann algebra and $\pi:\mathcal{A}\rightarrow\mathcal{M}$ is a unital $\ast
$-homomorphism. Then there is a unique $\ast$-homomorphism $\hat{\pi
}:\mathcal{A}^{\# \#}\rightarrow\mathcal{M}$ that is weak*-weak* continuous
(see \cite{KR}). \bigskip

\begin{lemma}
\label{commutative} Suppose $\left(  X,d\right)  $ is a compact metric space,
$\mathcal{M}$ is a sigma-finite von Neumann algebra, and $\pi,\rho:C\left(
X\right)  \rightarrow\mathcal{M}$, $\rho\in\mathcal{U}_{\mathcal{M}}\left(
\pi\right)  ^{-}$. Then there is a sequence $\mathcal{F}_{1},\mathcal{F}%
_{2},\ldots$ of finite disjoint collections of nonempty Borel sets such that

\begin{enumerate}
\item $\sum_{E\in\mathcal{F}_{n}}\hat{\pi}\left(  \chi_{E}\right)  =\sum
_{E\in\mathcal{F}_{n}}\hat{\rho}\left(  \chi_{E}\right)  =1,$

\item $\left\{  \hat{\pi}\left(  \chi_{E}\right)  :E\in\mathcal{F}%
_{n}\right\}  \subset$ \textrm{sp}$\left(  \left\{  \hat{\pi}\left(  \chi
_{F}\right)  :F\in\mathcal{F}_{n+1}\right\}  \right)  $ and $\left\{
\hat{\rho}\left(  \chi_{E}\right)  :E\in\mathcal{F}_{n}\right\}  \subset$
\textrm{sp}$\left(  \left\{  \hat{\rho}\left(  \chi_{F}\right)  :F\in
\mathcal{F}_{n+1}\right\}  \right)  ,$

\item For every $E\in\mathcal{F}_{n}$, and
\[
\text{\textrm{diam}}\left(  E\right)  <1/n\text{ .}%
\]

\item For every $E\in\cup_{n\in\mathbb{N}}\mathcal{F}_{n}$ $\hat{\pi}\left(
\chi_{E}\right)  $ and $\hat{\rho}\left(  \chi_{E}\right)  $ are Murray von
Neumann equivalent.
\end{enumerate}
\end{lemma}

\begin{proof}
Let \textrm{Bor}$\left(  X\right)  $ be the C*-algebra with the supremum norm.
We then have%
\[
C\left(  X\right)  \subset\text{\textrm{Bor}}\left(  X\right)  \subset
C\left(  X\right)  ^{\# \#}%
\]
and $\hat{\pi}|_{\text{\textrm{Bor}}\left(  X\right)  }$, $\hat{\rho
}|_{\text{\textrm{Bor}}\left(  X\right)  }$ are unital $\ast$-homomorphisms.

Let $\Sigma=\left\{  U\subset X:U\text{ is open and }\hat{\pi}\left(
\chi_{\bar{U}\backslash U}\right)  =\hat{\rho}\left(  \chi_{\bar{U}\backslash
U}\right)  =0\right\}  $. It is easily shown that if $U,V\in\Sigma$, then
$U\backslash\bar{V},$ $U\cup V$, $U\cap V\in\Sigma$. Moreover, if $a\in X$ and
$S\left(  a,r\right)  =\left\{  x\in X:d\left(  a,x\right)  =r\right\}  $ for
all $r>0$, it follows from the fact that $\mathcal{M}$ is $\sigma$-finite that
if $E_{a}=\left\{  r\in\left(  0,\infty\right)  :\hat{\pi}\left(
\chi_{S\left(  a,r\right)  }\right)  =\hat{\rho}\left(  \chi_{S\left(
a,r\right)  }\right)  =0\right\}  $, then $\left(  0,\infty\right)  \backslash
E_{a}$ is countable.

We can assume that $\mathrm{diam}\left(  X\right)  <1$ and we can let
$\mathcal{F}_{1}=\left\{  X\right\}  $.

Suppose $n\in\mathbb{N}$ and $\mathcal{F}_{n}$ has been defined.

For each $a\in X,$ there is an $r_{a}\in E_{a}\cap\left(  0,\frac{1}{2\left(
n+1\right)  }\right)  $. Since $X$ is compact and $\left\{
\text{\textrm{ball}}\left(  a,r_{a}\right)  :a\in X\right\}  $ is an open
cover with sets in $\Sigma$, there is a finite subcover $\left\{  U_{1}%
,\ldots,U_{s}\right\}  .$ We let $V_{1}=U_{1},$ and $V_{k}=U_{k}\backslash
\cup_{1\leq j<k}\bar{U}_{j}$ for $1<k\leq s$. Then $\left\{  V_{1},\ldots
V_{s}\right\}  $ is a disjoint family of open sets in $\Sigma$ with union $V$
such that
\[
\hat{\pi}\left(  \chi_{V}\right)  =\hat{\rho}\left(  \chi_{V}\right)
=1\text{.}%
\]
We now let
\[
\mathcal{F}_{n+1}=\left\{  V_{j}\cap W:1\leq j\leq s,W\in\mathcal{F}_{n}%
,V_{j}\cap W\neq\varnothing\right\}  .
\]

If $U\subset X$ is open and nonempty, then there is a continuous
$f:X\rightarrow\left[  0,1\right]  $ such that $f\left(  x\right)  =0$ if and
only if $x\in X\backslash U$. Thus the sequence $f^{1/n}\uparrow\chi_{U}$,
which means%
\[
f^{1/n}\rightarrow\chi_{U}%
\]
weak* in $C\left(  X\right)  ^{\#\#}$. Thus $\pi\left(  f\right)
^{1/n}\uparrow\hat{\pi}\left(  \chi_{U}\right)  $ and $\rho\left(  f\right)
^{1/n}\uparrow\hat{\rho}\left(  \chi_{U}\right)  $ in the weak* topology. Thus
$\hat{\pi}\left(  \chi_{U}\right)  $ is the projection onto the closure of the
range of $\pi\left(  f\right)  $ and $\hat{\rho}\left(  \chi_{U}\right)  $ is
the projection onto the closure of the range of $\rho\left(  f\right)  $. It
follows from Corollary \ref{rank} that $\hat{\pi}\left(  \chi_{U}\right)  $
and $\hat{\rho}\left(  \chi_{U}\right)  $ are Murray von Neumann equivalent.
\end{proof}

\bigskip

\begin{theorem}
\label{abelian}Suppose $\mathcal{A}$ is a separable unital commutative
C*-algebra and $\mathcal{M}$ is a von Neumann algebra. If $\pi\in$
\textrm{Rep}$\left(  \mathcal{A},\mathcal{M}\right)  $ then $\mathcal{U}%
_{\mathcal{B}}\left(  \pi\right)  ^{-}$ is path-connected. In fact, for every
$\rho\in\mathcal{U}_{\mathcal{M}}\left(  \pi\right)  ^{-}$ there is a strong
internal path from $\pi$ to $\rho$.
\end{theorem}

\begin{proof}
Suppose $\rho\in\mathcal{U}_{\mathcal{M}}\left(  \pi\right)  ^{-}$. Since
$\mathcal{A}$ is separable, there is a sequence $\left\{  U_{n}\right\}
\in\mathcal{U}_{\mathcal{M}}$ such that, for every $a\in\mathcal{A}$,%
\[
\lim_{n\rightarrow\infty}\left\Vert U_{n}^{\ast}\pi\left(  a\right)
U_{n}-\rho\left(  a\right)  \right\Vert =0.
\]
Let $\mathcal{N}=W^{\ast}\left(  \pi\left(  \mathcal{A}\right)  \cup
\rho\left(  \mathcal{A}\right)  \cup\left\{  U_{1},U_{2},\ldots\right\}
\right)  $. Then $\mathcal{N}$ is a countably generated von Neumann algebra,
and $\pi,\rho:\mathcal{A}\rightarrow\mathcal{N}$. Hence we can write%
\[
\mathcal{N}=\sum_{i\in I}^{\oplus}\mathcal{N}_{i},
\]
where each $\mathcal{N}_{i}$ acts on a separable Hilbert space, and we can
write%
\[
\pi=\sum_{i\in I}^{\oplus}\pi_{i}\text{ and }\rho=\sum_{i\in I}^{\oplus}%
\rho_{i}.\text{ }%
\]
We also have
\[
\hat{\pi}=\sum_{i\in I}^{\oplus}\hat{\pi}_{i}\text{ and }\hat{\rho}=\sum_{i\in
I}^{\oplus}\hat{\rho}_{i}.
\]

For each $i\in I$, we can choose a sequence $\mathcal{F}_{n,i}$ of families of
nonempty open subsets as in Lemma \ref{commutative}. Since, for each $i\in I$
and each $n\in N$ and each $E\in\mathcal{F}_{n,i}$ we know $\hat{\pi}%
_{i}\left(  \chi_{E}\right)  $ and $\hat{\rho}_{i}\left(  \chi_{E}\right)  $
are Murray von Neumann equivalent in $\mathcal{N}_{i}$ and since%
\[
\sum_{E\in\mathcal{F}_{n}}\hat{\pi}_{i}\left(  \chi_{E}\right)  =\sum
_{E\in\mathcal{F}_{n}}\hat{\rho}_{i}\left(  \chi_{E}\right)  =1,
\]
there is a unitary $U_{n,i}\in\mathcal{N}_{i}$ such that%
\[
U_{n,i}^{\ast}\hat{\pi}_{i}\left(  \chi_{E}\right)  U_{n,i}=\hat{\rho}%
_{i}\left(  \chi_{E}\right)
\]
for every $E\in\mathcal{F}_{n,i}$. For each $n\in\mathbb{N}$, let $U_{n}%
=\sum_{i\in I}^{\oplus}U_{n,i}$ for each $i\in I$, and let $\mathcal{D}%
_{n}=\sum_{i\in I}^{\oplus}$\textrm{sp}$\left(  \left\{  \hat{\pi}_{i}\left(
\chi_{E}\right)  :E\in\mathcal{F}_{n,i}\right\}  \right)  $. Since
$U_{n}U_{n+1}^{\ast}\in\mathcal{D}_{n}^{\prime}$, we know from the proof of
Lemma \ref{stitch} that the map $n\mapsto U_{n}$ on $\mathbb{N}$ extends to a
continuous map $t\mapsto U_{t}=\sum_{i\in I}^{\oplus}U_{t,i}$ such that
$U_{0}=1,$ and such that, for every $n\in\mathbb{N}$, for every $i\in I$,
every $n\leq t<\infty$, and every $E\in\mathcal{F}_{n,i}$%
\[
U_{t,i}^{\ast}\hat{\pi}_{i}\left(  \chi_{E}\right)  U_{t,i}=U_{n,i}^{\ast}%
\hat{\pi}_{i}\left(  \chi_{E}\right)  U_{n,i}=\hat{\rho}\left(  \chi
_{E}\right)  .
\]
Suppose $f\in C\left(  X\right)  $ and $\varepsilon>0$. Since $f$ is uniformly
continuous, there is a positive integer $n_{0}$ such that, if $x,y\in X$ and
$d\left(  x,y\right)  <1/n_{0}$, then $\left\vert f\left(  x\right)  -f\left(
y\right)  \right\vert <\varepsilon/2$.

For each $i\in I$  and all $E\in\mathcal{F}_{n_{0},i}$ we choose
$x_{i,n_{0},E}\in E$. Since \textrm{diam}$\left(  E\right)  <1/n_{0}$, we then
have%
\[
\left\Vert \left[  f-f\left(  x_{n_{0},i,E}\right)  \right]  \chi
_{E}\right\Vert <\varepsilon/2,
\]
so
\[
\left\Vert \pi_{i}\left(  f\right)  -\sum_{E\in\mathcal{F}_{n_{0,i}}}f\left(
x_{n_{0},i,E}\right)  \hat{\pi}_{i}\left(  \chi_{E}\right)  \right\Vert
\leq\varepsilon/2,
\]
and%
\[
\left\Vert \rho_{i}\left(  f\right)  -\sum_{E\in\mathcal{F}_{n_{0,i}}}f\left(
x_{n_{0},i,E}\right)  \hat{\rho}_{i}\left(  \chi_{E}\right)  \right\Vert
\leq\varepsilon/2\text{.}%
\]
Thus, for $t\geq n_{0}$, we have%
\[
\left\Vert U_{t}^{\ast}\pi\left(  f\right)  U_{t}-\rho\left(  f\right)
\right\Vert =
\]%
\[
\sup_{i\in I}\left\Vert U_{t,i}^{\ast}\pi_{i}\left(  f\right)  U_{t,i}%
-\rho_{i}\left(  f\right)  \right\Vert \leq
\]%
\[
\sup_{i\in I}\left\Vert U_{t,i}^{\ast}\left[  \pi_{i}\left(  f\right)
-\sum_{E\in\mathcal{F}_{n_{0,i}}}f\left(  x_{n_{0},i,E}\right)  \hat{\pi}%
_{i}\left(  \chi_{E}\right)  \right]  U_{t,i}\right\Vert
\]%
\[
+\sup_{i\in I}\left\Vert \sum_{E\in\mathcal{F}_{n_{0},i}}f\left(
x_{n_{0},i,E}\right)  \left[  U_{t,i}^{\ast}\hat{\pi}_{i}\left(  \chi
_{E}\right)  U_{t,i}-\hat{\rho}_{i}\left(  \chi_{E}\right)  \right]
\right\Vert
\]%
\[
+\sup_{i\in I}\left\Vert \sum_{E\in\mathcal{F}_{n_{0,i}}}f\left(
x_{n_{0},i,E}\right)  \hat{\rho}_{i}\left(  \chi_{E}\right)  -\rho_{i}\left(
f\right)  \right\Vert
\]%
\[
\leq\varepsilon/2+0+\varepsilon/2\text{ }=\varepsilon\text{.}%
\]
Thus, the map $t\mapsto U_{t}$ is continuous on $[1,\infty)$, and, for every
$f\in C\left(  X\right)  $,%
\[
\lim_{t\rightarrow\infty}\left\Vert U_{t}\pi\left(  f\right)  U_{t}^{\ast
}-\rho\left(  f\right)  \right\Vert =0\text{.}%
\]

\end{proof}

\begin{corollary}
Suppose $\mathcal{A}$ is a separable unital homogeneous C*-algebra and
$\mathcal{M}$ is a von Neumann algebra. If $\pi\in$ \textrm{Rep}$\left(
\mathcal{A},\mathcal{M}\right)  $ then $\mathcal{U}_{\mathcal{B}}\left(
\pi\right)  ^{-}$ is path-connected. In fact, for every $\rho\in
\mathcal{U}_{\mathcal{M}}\left(  \pi\right)  ^{-}$ there is a strong internal
path from $\pi$ to $\rho$.
\end{corollary}

\end{document}